\documentclass[11pt,reqno]{amsart}
\title[Qualitative Uncertainty Principle for Gabor transform]{Qualitative Uncertainty Principle for Gabor transform}
\usepackage{amsmath,amssymb,amsthm,amsrefs,mathrsfs}
\numberwithin{equation}{section}
\usepackage{enumerate}
\usepackage{tikz}
\usetikzlibrary{fit,shapes,snakes,positioning,calc,patterns}
\pgfdeclarelayer{background}
\pgfdeclarelayer{foreground}
\pgfsetlayers{background,main,foreground} 
 
\usepackage{tikz}
\usepackage{tabularx}
\usepackage{cellspace}
\theoremstyle{definition}
\newtheorem{thm}{\sc Theorem}[section]

\newtheorem{lem}[thm]{\sc Lemma}
\newtheorem{rem}[thm]{\sc Remark}

\newtheorem{defi}[thm]{\sc Definition}

\newcommand{\R}{\mathbb{R}}

\newcommand{\N}{\mathbb{N}}

\newcommand{\C}{\mathbb{C}}

\newcommand{\ch}{\mathcal{H}}

\DeclareMathOperator*{\tr}{tr}

\allowdisplaybreaks
\begin{document}
\begin{abstract}
We discuss the Qualitative Uncertainty Principle for Gabor transform on certain classes of the locally compact groups, like abelian groups, $\R^n\times K$, $K \ltimes \R^n$ where $K$ is compact group. We shall also prove a weaker version of Qualitative Uncertainty Principle for Gabor transform in case of compact groups.
\end{abstract}
\author[A. Bansal]{ASHISH BANSAL}
\address{Department of Mathematics, Keshav Mahavidyalaya (University of Delhi), H-4-5 Zone, Pitampura, Delhi, 110034, India.}
\email{abansal@keshav.du.ac.in}
\author[A. Kumar]{AJAY KUMAR}
\address{Department of Mathematics, University of Delhi, Delhi, 110007, India.}
\email[Corresponding author]{akumar@maths.du.ac.in}
\maketitle
\section{Introduction}
\noindent Let $G$ be a second countable, unimodular, locally compact group of type I with the dual space $\widehat{G}$. Let $m$ denote the Haar measure on $G$ and $\mu$ denote the Plancherel measure on $\widehat{G}$. For $f \in L^1(G)$, the Fourier transform $\widehat{f}$ is defined as the operator 
\begin{flalign*}
\widehat{f}(\gamma)=\int_{G}{f(x)\ \gamma(x)^\ast}\ dm_G(x).
\end{flalign*}
Let us define 
\begin{flalign*}
A_f=\{x\in G : f(x)\neq 0\}\ \text{and}\ B_{\widehat{f}}=\{\gamma \in \widehat{G} : \widehat{f}(\gamma)\neq 0\}.
\end{flalign*} 
Uncertainty principles have been studied extensively in the past fifty years. Although there is a variety of uncertainty principles, the common idea communicated by them is that a non-zero function and its Fourier transform cannot both be sharply localized. The qualitative uncertainty principle for Fourier transform can be stated as follows: 
\begin{flalign*}
\text{If $f\in L^1(G)$ satisfies $m(A_f)<\infty$ and $\mu(B_f)<\infty$, then $f=0$ a.e.}
\end{flalign*}
\noindent The QUP for $\R^n$ was proved by Benedicks \cite{Ben:85}. The principle has been generalized for several classes for locally compact groups by Echterhoff, Kaniuth and Kumar \cite{Ech:Kan:Kum:91}. For more details, refer to the survey \cite{Fol:Sit:97}.\\
The representation of $f$ as a function of $x$ is usually called \textit{time-representation}, whereas the representation of the Fourier transform $\hat{f}$ as a function of $\omega$ is called \textit{frequency-representation}. The Fourier transform is commonly used for analyzing the frequency properties of a given signal. After transforming a signal using Fourier transform, the information about time is lost and it is hard to tell where a certain frequency occurs. This problem can be countered by using \textit{joint time-frequency representation}, i.e., Gabor transform. It uses a window function to localize the Fourier transform, then shift the window to another position, and so on. This property of the Gabor transform provides the local aspect of the Fourier transform with time resolution equal to the size of the window.\\
Let $\psi \in L^2(\R)$ be a fixed non-zero function usually called a \textit{window function}. The Gabor transform of a function $f\in L^2(\R)$ with respect to the window function $\psi$ is defined by
\begin{flalign*}
G_\psi f : \R \times \widehat{\R} \rightarrow \C
\end{flalign*}
such that
\begin{flalign*}
G_\psi f(t,\omega)=\int_{\R}{f(x)\ \overline{\psi(x-t)}\ e^{-2\pi i \omega x}}\ dx,
\end{flalign*}
for all $(t,\omega)\in \R \times \widehat{\R}$. \\
In \cite{Wil:00}, it has been proved that for $f\in L^2(\R)\setminus \{0\}$ and a window function $\psi$, the support of $G_\psi f$ is a set of infinite Lebesgue measure.\\
The continuous Gabor transform for second countable, unimodular and type I group has been defined in \cite{Far:Kam:12}. A brief description is given in section 2. We will be interested in the following so called qualitative uncertainty principle for Gabor transform:
\begin{flalign*}
\text{If $f\in L^2(G)$ and $\psi$ is a window function satisfying\ \ \ }\\
\text{$(m\times \mu)(\{(x,\gamma):G_\psi f(x,\gamma)\neq 0\})<\infty$, then $f=0$ a.e.} 
\end{flalign*}
In section 3, we shall prove a necessary and sufficient condition for a second countable, locally compact, abelian group to have QUP. In section 4, for a second countable, locally compact, unimodular, type I group $G$ and a closed, normal subgroup $H$ of $G$ such that $G/H$ is compact, we prove that if $H$ has QUP, then so does $G$. In the last section, we shall prove the necessary and sufficient condition for a weaker form of QUP for Gabor transform to be true for a compact group $G$.
\section{Continuous Gabor Transform} 
\noindent Let $G$ be a second countable, unimodular group of type I. Let $dx$ denotes the Haar measure on $G$ and $d\pi$ the Plancherel measure on $\widehat{G}$. For each $(x,\pi)\in G\times \widehat{G}$, we define
\begin{flalign*}
\ch_{(x,\pi)}=\pi(x)\text{HS}(\ch_\pi),
\end{flalign*}
where $\pi(x)\text{HS}(\ch_\pi)=\{\pi(x)T:T\in \text{HS}(\ch_\pi)\}$ and $\ch_{(x,\pi)}$ forms a Hilbert space with the inner product given by
\begin{flalign*}
\langle \pi(x)T,\pi(x)S\rangle_{\ch_{(x,\pi)}}=\tr{(S^\ast T)}=\langle T,S\rangle_{\text{HS}(\ch_\pi)}.
\end{flalign*}
Also, $\ch_{(x,\pi)}=\text{HS}(\ch_\pi)$ for all $(x,\pi)\in G\times \widehat{G}$. The family $\{\ch_{(x,\pi)}\}_{(x,\pi)\in G\times \widehat{G}}$ of Hilbert spaces indexed by $G\times \widehat{G}$ is a field of Hilbert spaces over $G\times \widehat{G}$. Let $\ch^2(G \times \widehat{G})$ denote the direct integral of $\{\ch_{(x,\pi)}\}_{(x,\pi)\in G\times \widehat{G}}$ with respect to the product measure $dx\ d\pi$, i.e., the space of all measurable vector fields $F$ on $G\times \widehat{G}$ such that
\begin{flalign*}
\|F\|_{\ch^2(G\times \widehat{G})}^2=\int_{G\times \widehat{G}}{\|F(x,\pi)\|_{(x,\pi)}^2}\ dx\ d\pi<\infty.
\end{flalign*}
It can be easily verified that $\ch^2(G \times \widehat{G})$ forms a Hilbert space with the inner product given by
\begin{flalign*}
\langle F,K\rangle_{\ch^2(G\times \widehat{G})}=\int_{G\times \widehat{G}}{\tr{[F(x,\pi)K(x,\pi)^\ast]}}\ dx\ d\pi.
\end{flalign*}
Let $f\in C_c(G)$, the set of all continuous complex-valued functions on $G$ with compact supports and $\psi$ a fixed non-zero function in $L^2(G)$ usually called window function. For $(x,\pi) \in G\times \widehat{G}$, the continuous \textit{Gabor Transform} of $f$ with respect to the window function $\psi$ can be defined as a measurable field of operators on $G\times \widehat{G}$ by
\begin{flalign}
&&G_\psi f(x,\pi)&:=\int_{G}{f(y)\ \overline{\psi(x^{-1}y)}\ \pi(y)^\ast}\ dy. \label{opvalued}&
\end{flalign}
The operator-valued integral \eqref{opvalued} is considered in the weak-sense, i.e., for each $(x,\pi) \in G\times \widehat{G}$ and $\xi, \eta\in \ch_\pi$, we have
\begin{flalign*}
\langle G_\psi f(x,\pi)\xi,\eta\rangle &=\int_{G}{f(y)\ \overline{\psi(x^{-1}y)}\ \langle\pi(y)^\ast \xi,\eta\rangle}\ dy. 
\end{flalign*}
For each $x\in G$, define $f^\psi_{x}:G \rightarrow \C$ by
\begin{flalign}
f^\psi_x(y):=f(y)\ \overline{\psi(x^{-1}y)}.  \label{function}
\end{flalign}
Since, $f \in C_c(G)$ and $\psi \in L^2(G)$, we have $f^\psi_x\in L^1(G)\cap L^2(G)$, for all $x \in G$. The Fourier transform is given by
\begin{flalign*}
&&\widehat{f^\psi_x}(\pi)&=\int_{G}{f^\psi_x(y)\ \pi(y)^\ast}\ dy =\int_{G}{f(y)\ \overline{\psi(x^{-1}y)}\ \pi(y)^\ast}\ dy =G_\psi f(x,\pi). &
\end{flalign*}
Also, using Plancherel theorem \cite[Theorem $7.44$]{Fol:94}, we see that $\widehat{f^\psi_x}(\pi)$ is a Hilbert-Schmidt operator for almost all $\pi\in \widehat{G}$. Therefore, $G_\psi f(x,\pi)$ is a Hilbert-Schmidt operator for all $x \in G$ and for almost all $\pi\in \widehat{G}$. As in \cite{Far:Kam:12}, for $f \in C_c(G)$ and a window function $\psi \in L^2(G)$, we have
\begin{flalign}
\|G_\psi f\|_{\ch^2(G\times \widehat{G})} =\|\psi\|_2\ \|f\|_2. \label{GT-norm}
\end{flalign}
It means that the continuous Gabor transform $G_\psi : C_c(G) \rightarrow \ch^2(G\times \widehat{G})$ defined by $f \mapsto G_\psi f$ is a multiple of an isometry. So, we can extend $G_\psi$ uniquely to a bounded linear operator from $L^2(G)$ into a closed subspace $H$ of $\ch^2(G\times \widehat{G})$ which we still denote by $G_\psi$ and this extension satisfies \eqref{GT-norm} for each $f\in L^2(G)$. It follows from \cite{Ban:Kum:15(2)} that for $f \in L^2(G)$ and a window function $\psi \in L^2(G)$, we have $G_\psi f(x,\pi)=\widehat{f^\psi_x}(\pi)$.
\section{QUP for Gabor transform}
\noindent In this section $G$ will be second countable, locally compact, abelian group with Haar measure $m$. Let $\widehat{G}$ be the dual group with Plancherel measure $\mu$. Before discussing the QUP for Gabor transform on $G$, we shall first establish some important properties of Gabor transform.  
\begin{lem}\label{invariance}
For $f\in L^2(G)$ and a window function $\psi$, we have
\renewcommand{\labelenumi}{(\roman{enumi})}
\begin{enumerate}
\item $G_\psi({}_{x_0}f)(x,\gamma)=\gamma(x_0)\ G_\psi f(x_0x,\gamma)$, for $x_0,x\in G$ and $\gamma \in \widehat{G}$.
\item $G_\psi(\sigma f)(x,\gamma)=G_\psi f(x,\sigma^{-1}\gamma)$, for $x\in G$ and $\sigma,\gamma \in \widehat{G}$.
\end{enumerate}
\end{lem}
\begin{proof}
\renewcommand{\labelenumi}{(\roman{enumi})}
\begin{enumerate}
\item For $x_0,x\in G$ and $\gamma \in \widehat{G}$, we have
\begin{flalign*}
&&G_\psi({}_{x_0}f)(x,\gamma)&=\int_{G}{f(x_0y)\ \overline{\psi(x^{-1}y)}\ \gamma(y^{-1})}\ dm(y) &\\
&&&=\int_{G}{f(y)\ \overline{\psi(x^{-1}x_0^{-1}y)}\ \gamma(y^{-1}x_0)}\ dm(y) &\\
&&&=\gamma(x_0)\int_{G}{f(y)\ \overline{\psi((x_0x)^{-1}y)}\ \gamma(y^{-1})}\ dm(y) &\\
&&&=\gamma(x_0)\ G_\psi f(x_0x,\gamma). &
\end{flalign*}
\item For $x\in G$ and $\sigma,\gamma \in \widehat{G}$, we observe that
\begin{flalign*}
&&G_\psi(\sigma f)(x,\gamma)&=\int_{G}{(\sigma f)(y)\ \overline{\psi(x^{-1}y)}\ \gamma(y^{-1})}\ dm(y) &\\
&&&=\int_{G}{f(y)\ \overline{\psi(x^{-1}y)}\ (\sigma^{-1}\gamma)(y^{-1})}\ dm(y) &\\
&&&=G_\psi f(x,\sigma^{-1}\gamma). &
\end{flalign*}
\end{enumerate}
\end{proof}
\begin{defi}
If $\ch$ be a Hilbert space of $\C$-valued functions defined on a non-empty set $X$. A function $k: X \times X \rightarrow \C$ is called a \textit{reproducing kernel} of $\ch$ if it satisfies
\renewcommand{\labelenumi}{(\roman{enumi})}
\begin{enumerate}
\item $k_x \in \ch$, for all $x \in X$, where $k_x(y)=k(y,x)$ for all $y\in X$.
\item $\langle{f,k_x}\rangle_{\ch}=f(x)$, for all $x \in X$ and $f \in \ch$. 
\end{enumerate}
\end{defi}
\noindent One can easily verify that if reproducing kernel of $\ch$ exists, then it is unique.
\begin{defi}
A Hilbert space $\ch$ is a reproducing kernel Hilbert space (r.k.H.s.) if the evaluation functionals $F_t: \ch \rightarrow \C$ given by $F_t(f)=f(t)$ for all $f \in \ch$, are bounded. 
\end{defi}
\noindent We can observe that a Hilbert space $\ch$ is a r.k.H.s. if and only if $\ch$ has a reproducing kernel.
\noindent Let $\psi$ be a window function. Then, we define
\begin{flalign*}
G_\psi(L^2(G)) =\{G_\psi f : f \in L^2(G)\}\subseteq L^2(G\times \widehat{G}).
\end{flalign*}
This space satisfies a very important property as shown in the following lemma:
\begin{lem}\label{r.k.h.s.-GT}
$G_\psi(L^2(G))$ is a r.k.H.s. with pointwise bounded kernel.
\end{lem}
\begin{proof}
Define $K_\psi : (G \times \widehat{G})\times (G \times \widehat{G}) \rightarrow \C$ by
\begin{flalign*}
K_\psi(x',\gamma',x,\gamma)=\dfrac{1}{\|\psi\|_2^2}\ \langle{\psi_{(x',\gamma')},\psi_{(x,\gamma)}}\rangle_{L^2(G)},
\end{flalign*} 
where $\psi_{(x,\gamma)}(y)=\psi(x^{-1}y)\ \gamma(y)$, and let 
\begin{flalign*}
K_\psi^{(x',\gamma')}(x,\gamma)=K_\psi(x',\gamma',x,\gamma).
\end{flalign*}
For all $(x',\gamma')\in G\times \widehat{G}$, we have
\begin{flalign*}
&&K_\psi^{(x',\gamma')}(x,\gamma)&=\dfrac{1}{\|\psi\|_2^2}\ \int_{G}{\psi_{(x',\gamma')}(y)\ \overline{\psi_{(x,\gamma)}(y)}}\ dy &\\
&&&=\dfrac{1}{\|\psi\|_2^2}\ \int_{G}{\psi_{(x',\gamma')}(y)\ \overline{\psi(x^{-1}y)}\ \gamma(y^{-1})}\ dy &\\
&&&=G_{\psi}\left(\dfrac{1}{\|\psi\|_2^2}\ \psi_{(x',\gamma')}\right)(x,\gamma) &\\
&&&=G_{\psi}g(x,\gamma), &
\end{flalign*} 
where $g=\dfrac{1}{\|\psi\|_2^2}\ \psi_{(x',\gamma')}\in L^2(G)$. So $K_\psi^{(x',\gamma')}=G_{\psi}g \in G_\psi(L^2(G))$.\\
For all $(x',\gamma') \in G\times \widehat{G}$ and $f \in L^2(G)$, we have
\begin{flalign*}
&&\langle{G_\psi f,K_\psi^{(x',\gamma')}}\rangle_{L^2(G\times \widehat{G})}&=\dfrac{1}{\|\psi\|_2^2}\int_{G\times \widehat{G}}{G_\psi f(x,\gamma)\ \overline{\langle{\psi_{(x',\gamma')},\psi_{(x,\gamma)}}\rangle_{L^2(G)}}}\ dy &\\
&&&=\dfrac{1}{\|\psi\|_2^2}\int_{G\times \widehat{G}}{G_\psi f(x,\gamma)\ \overline{G_\psi (\psi_{(x',\gamma')})(x,\gamma)}}\ dy &\\
&&&=\langle{f,\psi_{(x',\gamma')}}\rangle_{L^2(G)}=G_{\psi}f(x',\gamma'). &
\end{flalign*}
Thus, $G_\psi(L^2(G))$ is a r.k.H.s. with reproducing kernel $K_\psi$ satisfying
\begin{flalign*}
|K_\psi(x',\gamma',x,\gamma)|&=\dfrac{1}{\|\psi\|_2^2}\ |\langle{\psi_{(x',\gamma')},\psi_{(x,\gamma)}}\rangle_{L^2(G)}| &\\
&\leq \dfrac{1}{\|\psi\|_2^2}\ \|\psi_{(x',\gamma')}\|\ \|\psi_{(x,\gamma)}\| = \dfrac{1}{\|\psi\|_2^2}\ \|\psi\|_2\ \|\psi\|_2 =1.&
\end{flalign*} 
Hence, the reproducing kernel is pointwise bounded by $1$.
\end{proof}
\begin{thm}\label{supp-GT}
Let $G$ be a second countable, locally compact, abelian group. If $f\in L^2(G)$ and $\psi$ is a window function, then QUP for Gabor transform holds if and only if the identity component $G_0$ of $G$ is non-compact.
\end{thm}
\begin{proof}
Suppose that $G$ has non-compact identity component $G_0$. \\
Let $f\in L^2(G)\setminus \{0\}$ be arbitrary. In order to show that the measure of the set $\{(x,\gamma):G_\psi f(x,\gamma) \neq 0\}$ is infinite, it suffices to show that for arbitrary set $M \subseteq G\times \widehat{G}$ of finite measure, we have
\begin{flalign}
G_\psi(L^2(G))\cap \{F\in L^2(G\times \widehat{G}) : F=\chi_M \cdot F\}=\{0\}. \label{main}
\end{flalign}
Let us assume, on the contrary, that there exists a non-trivial function $F_0$ such that for arbitrary set $M \subseteq G\times \widehat{G}$ of finite measure, we have
\begin{flalign*}
F_0 \in G_\psi(L^2(G))\cap \{F\in L^2(G\times \widehat{G}) : F=\chi_M \cdot F\}.
\end{flalign*}
Let $\epsilon >0$ be arbitrary and $M_0=\{(x,\gamma): F_0(x,\gamma) \neq 0\}\subseteq M$. Since $(m\times \mu)(M_0)>0$, by \cite[Proposition 1]{Hog:93} there exists $a^{(1)} \in (G\times \widehat{G})_0$ such that
\begin{flalign*}
(m\times \mu)(M)<(m\times \mu)(M\cup a^{(1)}M_0) <(m\times \mu)(M)+\dfrac{\epsilon}{2},
\end{flalign*} 
where $(G\times \widehat{G})_0=G_0\times (\widehat{G})_0$ denotes the identity component of $G\times \widehat{G}$. Then, we can write
\begin{flalign*}
a^{(1)}=(y^{(1)},\sigma^{(1)}),\ \text{where}\ y^{(1)}\in G_0,\ \sigma^{(1)}\in (\widehat{G})_0
\end{flalign*}
and 
\begin{flalign*}
a^{(1)}M_0=\{(y^{(1)}x,\sigma^{(1)}\gamma):(x,\gamma)\in M_0\}.
\end{flalign*}
Define
\begin{flalign*}
M_1:=M,\ \ M_2:=M\cup a^{(1)}M_0.
\end{flalign*}
Since $0<(m\times \mu)(M_2)<\infty$ and $a^{(1)}M_0 \subseteq M_2$ with $(m\times \mu)(a^{(1)}M_0)>0$, there exists $a^{(2)}=(y^{(2)},\sigma^{(2)})\in G_0\times (\widehat{G})_0$ such that
\begin{flalign*}
(m\times \mu)(M_2)<(m\times \mu)(M_2\cup a^{(2)}a^{(1)}M_0) <(m\times \mu)(M_2)+\dfrac{\epsilon}{2^2}.
\end{flalign*} 
Proceeding in this way, we get an increasing sequence $\{M_k\}_{k\geq 2}$ given by
\begin{flalign*}
M_k:=M_{k-1}\cup a^{(k-1)}\cdots a^{(2)}a^{(1)}M_0,
\end{flalign*}
where $a^{(j)}=(y^{(j)},\sigma^{(j)})\in G_0\times (\widehat{G})_0$ for all $j=1,2,\ldots,k-1$ satisfying
\begin{flalign}
&&(m\times \mu)(M_{k-1})<(m\times \mu)(M_k) &<(m\times \mu)(M_{k-1})+\dfrac{\epsilon}{2^{k-1}}. &\label{Mk}
\end{flalign} 
Let us now define
\begin{flalign*}
S=\bigcup_{k=1}^{\infty}{M_k}.
\end{flalign*}
\begin{flalign*}
\text{Then,}\, (m\times \mu)(S)&=\lim_{k\rightarrow \infty}{(m\times \mu)(M_k)}&\\*
&\leq \lim_{k\rightarrow \infty}{\left[(m\times \mu)(M_{k-1})+\dfrac{\epsilon}{2^{k-1}}\right]}&\\
&\leq \lim_{k\rightarrow \infty}{\left[(m\times \mu)(M)+\dfrac{\epsilon}{2}+\cdots+\dfrac{\epsilon}{2^{k-1}}\right]}&\\
&=(m\times \mu)(M)+\lim_{k\rightarrow \infty}{\left[\sum_{i=1}^{k-1}{\dfrac{\epsilon}{2^i}}\right]}&\\
&=(m\times \mu)(M)+\epsilon<\infty. &
\end{flalign*}
Consider the family $\{F_k\}_{k \in \N}$ of functions on $G \times \widehat{G}$ defined as follows:
\begin{flalign*}
F_1(x,\gamma):&=F_0(x,\gamma), \\
F_k(x,\gamma):&=\gamma((y^{(k-1)})^{-1}) F_{k-1}((y^{(k-1)})^{-1}x,(\sigma^{(k-1)})^{-1}\gamma),\ \text{for} \ k>2.
\end{flalign*}
We first show that $F_k\in G_\psi(L^2(G))$, for all $k\in \N$. This is proved by induction on $k$. For $k=1$, the result is trivially true. \\
Assume that $F_{k-1}=G_\psi(g_{k-1})$, for some $g_{k-1}\in L^2(G)$.\\
Then, using Lemma \ref{invariance}, we can write
\begin{flalign*}
&&F_k(x,\gamma)&=\gamma((y^{(k-1)})^{-1})\ G_\psi(g_{k-1})((y^{(k-1)})^{-1}x,(\sigma^{(k-1)})^{-1}\gamma) &\\
&&&=\gamma((y^{(k-1)})^{-1})\ G_\psi(\sigma^{(k-1)}g_{k-1})((y^{(k-1)})^{-1}x,\gamma) &\\
&&&=G_\psi({}_{(y^{(k-1)})^{-1}}(\sigma^{(k-1)}g_{k-1}))(x,\gamma) &\\
&&&=G_\psi(g_k)(x,\gamma), &
\end{flalign*}
where $g_k={}_{(y^{(k-1)})^{-1}}(\sigma^{(k-1)}g_{k-1})\in L^2(G)$ as $g_{k-1}\in L^2(G)$.
\begin{flalign*}
\text{Also,}\, \{(x,\gamma): F_k(x,\gamma) \neq 0\}&=\{(x,\gamma): F_{k-1}((a^{(k-1)})^{-1}(x,\gamma)) \neq 0\}&\\
&=\{a^{(k-1)}(y,\sigma): F_{k-1}(y,\sigma) \neq 0\}&\\
&=a^{(k-1)}\cdots a^{(2)}a^{(1)}\{(x,\gamma): F_0(x,\gamma) \neq 0\} &\\
&=a^{(k-1)}\cdots a^{(2)}a^{(1)}M_0 \subseteq M_k\subset S.
\end{flalign*}
Next we claim that the family $\{F_k\}_{k\geq 2}$ is linearly independent. Assume that there exists $k>2$ such that $F_k=\displaystyle\sum_{j=2}^{k-1}{b_jF_j}$, where $b_2,b_3,\ldots,b_{k-1}\in \C$ are suitably chosen constants. Then
\begin{flalign*}
a^{(k-1)}\cdots a^{(2)}a^{(1)}M_0&=\{(x,\gamma): F_k(x,\gamma) \neq 0\} &\\
&\subseteq \bigcup_{j=2}^{k-1}{\{(x,\gamma): F_j(x,\gamma) \neq 0\}} &\\
&=(a^{(1)}M_0)\cup (a^{(2)}a^{(1)}M_0) \cup \ldots \cup (a^{(k-1)}\cdots a^{(2)}a^{(1)}M_0) &\\
&\subseteq M_{k-1}, &
\end{flalign*}
which implies that $M_k=M_{k-1}$, which contradicts \eqref{Mk}.

\noindent Therefore, $\{F_k\}_{k\geq 2}$ is an infinite set of linearly independent functions with $\{(x,\gamma): F_k(x,\gamma) \neq 0\}\subseteq S$, where $(m\times \mu)(S)<\infty$.\\
By Lemma \ref{r.k.h.s.-GT}, $G_\psi(L^2(G))$ is a r.k.H.s. with pointwise bounded kernel, so by \cite[Lemma 3.1]{Wil:00} each subspace of $G_\psi(L^2(G))$ consisting of functions that are non-zero on a set of finite measure must be of finite dimension. This is a contradiction.\\
So $G_\psi(L^2(G))\cap \{F\in L^2(G\times \widehat{G}) : F=\chi_M \cdot F\}=\{0\}$ for arbitrary set $M \subseteq G\times \widehat{G}$ of finite measure.  \\
Hence, the set $\{(x,\gamma):G_\psi f(x,\gamma) \neq 0\}$ has infinite measure.\\
Conversely, suppose that for an arbitrary function $f\in L^2(G)\setminus \{0\}$, the set $\{(x,\gamma): G_\psi f(x,\gamma) \neq 0\}$ has infinite measure.  \\
Let, if possible, $G_0$ is compact. Then, the quotient group $G/G_0$ is totally disconnected and therefore has a compact open subgroup $K$. \\
Let $\pi : G \rightarrow G/G_0$ be the natural homomorphism. Then $\pi$ is continuous and open and there exists a compact open subset $C$ of $G$ such that $\pi(C)=K$. So $G_1=\pi^{-1}(K)=CG_0$ is a compact open subgroup of $G$. \\
Let $m(G_1)=\alpha>0$. Then $m_{G_1}=\alpha^{-1}(m|_{G_1})$ is a Haar measure on $G_1$ for which $m_{G_1}(G_1)=1$. \\
Define $f=\chi_{G_1}$ and $\psi=\chi_{G_1}$. Then
\begin{flalign*}
\|f\|_2^2=\|\psi\|_2^2 =\int_{G}{|\chi_{G_1}(x)|^2}\ dm(x)=m(G_1)=\alpha.
\end{flalign*}
Also, using \cite[Lemma 23.19]{Hew:Ros:63}, we have
\begin{flalign*}
G_\psi f(x,\gamma)&=\int_{G}{\chi_{G_1}(y)\ \overline{\chi_{G_1}(x^{-1}y)}\ \gamma(y^{-1})}\ dm(y) &\\
&=\int_{G_1}{\overline{\chi_{G_1}(x^{-1}y)}\ \gamma(y^{-1})}\ \alpha\ dm_{G_1}(y) &\\
&=\chi_{G_1}(x)\int_{G_1}{\gamma(y^{-1})}\ \alpha\ dm_{G_1}(y) &\\
&=\alpha \chi_{G_1}(x)\ \chi_{A(G_1)}(\gamma). &
\end{flalign*}
Therefore, $\{(x,\gamma): G_\psi f(x,\gamma) \neq 0\}=G_1 \times A(G_1)$.\\
Since $G_1$ is compact and $m(G_1)>0$, so $G_1$ is not locally null.\\
By \cite[23.24 (d), (e)]{Hew:Ros:63}, $A(G_1)$ is compact open subgroup.
\begin{flalign*}
\text{So,}\ 0<(m\times \mu)(\{(x,\gamma): G_\psi f(x,\gamma) \neq 0\}) &=(m\times \mu)(G_1 \times A(G_1)) &\\
&=m(G_1)\ \mu(A(G_1))<\infty, 
\end{flalign*}
which is a contradiction to the hypothesis. \\
Hence, $G_0$ is non-compact.
\end{proof}
\section{QUP for certain group extensions}
\noindent Throughout this section $G$ will be a second countable, unimodular, locally compact group of type I and $\widehat{G}$ the dual space of $G$. If $f$ is a function on $G$ and $y \in G$, we denote by $f_y|H$ the function on $H$ defined by
\begin{flalign*}
(f_y|H)(h)=f(hy),\ \text{for all}\ h \in H.
\end{flalign*}
We now prove  the following theorem.
\begin{thm}\label{QUP-GT}
Let $H$ be a closed, normal subgroup of $G$ such that $G/H$ is compact. If $H$ has QUP for Gabor transform, then so does $G$.
\end{thm}
\begin{proof}
Let $f \in L^2(G)$ and $\psi$ be a window function such that 
\begin{flalign*}
&&(m\times \mu)\{(x,\pi):G_\psi f(x,\pi)\neq 0\} &< \infty. &
\end{flalign*}
By Weil's formula, we obtain
\begin{flalign*}
&&\int_{G/H}\int_{H}\int_{\widehat{G}}{\chi_{\{(hx,\pi):G_\psi f(hx,\pi)\neq 0\}}(hx,\pi)}\ d\pi\ dh\ d\dot{x} &< \infty. &
\end{flalign*}
Therefore, there exists a zero set $K$ in $G$ such that for all $x\in G\setminus K$,
\begin{flalign}
&&\int_{H}\int_{\widehat{G}}{\chi_{\{(hx,\pi):G_\psi f(hx,\pi)\neq 0\}}(hx,\pi)}\ d\pi\ dh &< \infty. & \label{QUP-step1}
\end{flalign}
Fix $x\in G\setminus K$. For each $h \in H$, define
\begin{flalign*}
f_h^{({}_x\psi)} (y)=f(y)\ \overline{{}_x\psi(h^{-1}y)},\ \text{for all}\ y\in G.
\end{flalign*}
Then, $f_h^{({}_x\psi)} \in L^1(G)$ for all $h\in H$. \\
Also, for all $y\in G$, we observe that $f_{h}^{({}_x\psi)}=f_{hx}^\psi$.\\
Since $H$ is a closed unimodular subgroup of $G$, so by \cite[Theorem 1.2]{Ech:Kan:Kum:91} there exists a zero set $M_h$ in $G$ such that for every $y\in G\setminus M_h$ and every representation $\sigma$ of $H$, the function $\left(f_h^{({}_x\psi)}\right)_y|H \in L^1(H)$ and
\begin{flalign}
\mu_{\widehat{H}}\left(\left\{\sigma:\left(\left(f_{h}^{({}_x\psi)}\right)_y|H\right)\widehat{\ }(\sigma)\neq 0\right\}\right)\!\leq \! \mu\left(\left\{\pi:\left(f_{h}^{({}_x\psi)} \right)\widehat{\ }(\pi)\neq 0\right\}\right).\label{QUP-step2}
\end{flalign}
For all $k\in H$, we observe that $\left(\left(f_h^{({}_x\psi)}\right)_y|H\right)(k)=\left(f_y|H\right)_h^{(({}_x\psi)_y|H)}(k)$.
\begin{flalign}
\text{We have,}\ &\mu_{\widehat{H}}\left(\left\{\sigma:\left(\left(f_{h}^{({}_x\psi)}\right)_y|H\right)\widehat{\ }(\sigma)\neq 0\right\}\right) \nonumber&\\*
&=\mu_{\widehat{H}}\left(\left\{\sigma:\left(\left(f_y|H\right)_h^{(({}_x\psi)_y|H)}\right)\widehat{\ }(\sigma)\neq 0\right\}\right) \nonumber&\\
&=\int_{\widehat{H}}{\chi_{\left\{\sigma:G_{(({}_x\psi)_y|H)}\left(f_y|H\right)(h,\sigma)\neq 0\right\}}(\sigma)}\ d\sigma &\label{QUP-step3}
\end{flalign}
\begin{flalign}
\text{and}\ \mu\left(\left\{\pi:\left(f_{h}^{({}_x\psi)} \right)\widehat{\ }(\pi)\neq 0\right\}\right) &=\mu\left(\left\{\pi:\left(f_{hx}^\psi\right)\widehat{\ }(\pi)\neq 0\right\}\right)\nonumber&\\*
&=\int_{\widehat{G}}{\chi_{\left\{\pi:G_\psi f(hx,\pi)\neq 0\right\}}(\pi)}\ d\pi. &\label{QUP-step4}
\end{flalign}
From \eqref{QUP-step2}, \eqref{QUP-step3} and \eqref{QUP-step4}, we obtain
\begin{flalign*}
&&\int_{\widehat{H}}{\chi_{\left\{\sigma:G_{(({}_x\psi)_y|H)}\left(f_y|H\right)(h,\sigma)\neq 0\right\}}(\sigma)}\ d\sigma&\leq \int_{\widehat{G}}{\chi_{\left\{\pi:G_\psi f(hx,\pi)\neq 0\right\}}(\pi)}\ d\pi,  &
\end{flalign*}
for all $h\in H$ and $y\in G\setminus M_h$. Integrating both sides with respect to $h$, we get
\begin{flalign*}
\int_{H}\int_{\widehat{H}}{\chi_{\left\{\sigma:G_{(({}_x\psi)_y|H)}\left(f_y|H\right)(h,\sigma)\neq 0\right\}}(\sigma)}\ d\sigma\ dh&\leq \int_{H}\int_{\widehat{G}}{\chi_{\left\{\pi:G_\psi f(hx,\pi)\neq 0\right\}}(\pi)}\ d\pi\ dh, 
\end{flalign*}
for all $y\in G \setminus M$, where $M=\displaystyle\bigcup_{h\in H}{M_h}$. It implies
\begin{flalign*}
&\int_{H}\int_{\widehat{H}}{\chi_{\left\{(h,\sigma):G_{(({}_x\psi)_y|H)}\left(f_y|H\right)(h,\sigma)\neq 0\right\}}(h,\sigma)}\ d\sigma\ dh&\\
&\leq \int_{H}\int_{\widehat{G}}{\chi_{\left\{(hx,\pi):G_\psi f(hx,\pi)\neq 0\right\}}(hx,\pi)}\ d\pi\ dh   &\\
&<\infty. & [\text{Using \eqref{QUP-step1}}]
\end{flalign*}
Therefore, we have
\begin{flalign*}
(m_H \times \mu_H)\left(\left\{(h,\sigma):G_{(({}_x\psi)_y|H)}\left(f_y|H\right)(h,\sigma)\neq 0\right\}\right)<\infty,
\end{flalign*}
for all $y\in G \setminus M$. Since $H$ has QUP for Gabor transform, we see that $f_y|H=0$ a.e. for all $y\in G \setminus M$. Hence, by Weil's formula, $f=0$ a.e.
\end{proof}
\begin{rem}
Let $G$ contain an abelian, normal subgroup $H$ such that $G/H$ is compact and $H_0$ is non-compact, then $G$ satisfies QUP for Gabor transform. In particular, QUP for Gabor transform holds for Lie groups which are Moore group with non-compact identity component. 
\end{rem}
\begin{rem}
By Theorem \ref{QUP-GT}, QUP for Gabor transform holds for Euclidean motion group $SO(n) \ltimes~\R^n$. In fact, it holds for all the groups of the form $K \ltimes \R^n$, where $K$ is compact group. 
\end{rem} 
\begin{rem}
It can be seen easily that QUP for Gabor transform does not hold when $G$ is compact or discrete, by taking $f=\psi=\chi_G$ or $f=\psi=\chi_{\{e\}}$ respectively.
\end{rem}
\section{Weak QUP for Gabor transform}
\noindent Throughout this section, we shall assume that $G$ is a compact group. We shall normalize the Haar measure $m$ on $G$ so that $m(G)=1$.  We shall establish the necessary and sufficient condition for a weaker form of QUP for Gabor transform. We have the following result:
\begin{thm}
The following conditions are equivalent:
\renewcommand{\labelenumi}{(\roman{enumi})}
\begin{enumerate}
\item If $f\in L^2(G)$ and $\psi$ is a window function satisfying 
\begin{flalign*}
(m\times \mu)(\{(x,\gamma): G_\psi f(x,\gamma)\neq 0\}) <1,
\end{flalign*}
then $f=0$ a.e.
\item $G/G_0$ is abelian.
\end{enumerate}
\end{thm}
\begin{proof}
\textbf{(i) \boldmath$\Rightarrow$\unboldmath\ (ii)}: Suppose on the contrary that $G/G_0$ is non-abelian. \\
Since $G/G_0$ is totally disconnected, there exists an open normal subgroup $C$ of $G/G_0$ such that $(G/G_0)/C$ is non-abelian. \\
Let $H$ be the pre-image of $C$ in $G$. Then $G/H$ is finite and non-abelian.\\
We define $f=\chi_H$ and $\psi=\chi_H$. Then $f,\psi \in L^1(G) \cap L^2(G)$ and
\begin{flalign}
G_\psi f(x,\gamma) &=\int_{G}{\chi_H(y)\ \overline{\chi_H(x^{-1}y)}\ \gamma(y^{-1})}\ dm(y)\nonumber&\\*
&=\int_{H}{\overline{\chi_H(x^{-1}y)}\ \gamma(y^{-1})}\ dm(y) \nonumber&\\*
&=\left\{ \begin{array}{cl}
\widehat{f}(\gamma), & \mbox{if}\ x\in H \\[10pt] 
0, & \mbox{if}\ x \notin H
\end{array}\right.. &\label{WQUP-step1}
\end{flalign}
We define a function $g \in L^1(G/H)$ as $g=\chi_{\{H\}}$. \\
Then $f(x)=g(xH)$ for all $x \in G$. By \cite[Lemma 2.1]{Kut:03}, for $\gamma \in \widehat{G}$ and $\xi,\eta \in \ch_\gamma$, we have
\begin{flalign}
&&\langle{\widehat{f}(\gamma)\xi,\eta}\rangle &=\chi_{A(H,\widehat{G})}(\gamma)\ \langle{\widehat{g}(\gamma)\xi,\eta}\rangle. &\label{WQUP-step2}
\end{flalign} 
From \eqref{WQUP-step1} and \eqref{WQUP-step2}, we obtain
\begin{flalign*}
\langle{G_\psi f(x,\gamma)\xi,\eta}\rangle &=\left\{ \begin{array}{cl}
\chi_{A(H,\widehat{G})}(\gamma)\ \langle{\widehat{g}(\gamma)\xi,\eta}\rangle, & \mbox{if}\ x\in H \\[10pt] 
0, & \mbox{if}\ x \notin H
\end{array}\right. &\\
&=\left\{ \begin{array}{cl}
\displaystyle\sum_{yH \in G/H}{\chi_{\{H\}}(yH)\ \langle{\gamma(y^{-1})\xi,\eta}\rangle}, & \mbox{if}\ x\in H, \gamma \in A(H,\widehat{G}) \\[10pt] 
0, & \mbox{otherwise}
\end{array}\right. &\\
&=\left\{ \begin{array}{cl}
\langle{1_{\ch_\gamma}\xi,\eta}\rangle, & \mbox{if}\ x\in H, \gamma \in A(H,\widehat{G}) \\[10pt] 
0, & \mbox{otherwise}
\end{array}\right., &
\end{flalign*} 
which implies
\begin{flalign}
&(m\times \mu)(\{(x,\gamma): G_\psi f(x,\gamma)\neq 0\}) \nonumber&\\
&=(m\times \mu)(\{(x,\gamma): x \in H,\gamma\in A(H,\widehat{G})\}) \nonumber&\\
&=m(H)\ \mu(A(H,\widehat{G})) &\label{WQUP-step3}
\end{flalign}
Since $m(G)=1$, we have 
\begin{flalign}
m(H)=[G:H]^{-1}.  \label{WQUP-step4}
\end{flalign}
Also $G/H$ is non-abelian, there exists at least one $\gamma\in \widehat{G/H}$ such that $d_\gamma >1$. Since $H$ is a closed normal subgroup of $G$, by using \cite[Corollary 28.10]{Hew:Ros:63} we can identify $A(H,\widehat{G})$ with $\widehat{G/H}$. \\
As $G/H$ is finite group, by definition of Plancherel measure and \cite[Proposition 5.27]{Fol:94}, we have
\begin{flalign}
&&\mu(A(H,\widehat{G}))&=\mu(\widehat{G/H})=\sum_{\gamma\in \widehat{G/H}}{d_\gamma}<\sum_{\gamma\in \widehat{G/H}}{d_\gamma^2} =[G:H] & \label{WQUP-step5}
\end{flalign}
Combining \eqref{WQUP-step3}, \eqref{WQUP-step4} and \eqref{WQUP-step5}, we obtain 
\begin{flalign*}
&&(m\times \mu)(\{(x,\gamma): G_\psi f(x,\gamma)\neq 0\})&<1, &
\end{flalign*}
which is a contradiction to (i). Hence $G/G_0$ is abelian.\\
\textbf{(ii) \boldmath$\Rightarrow$\unboldmath\ (i)}: Suppose that $G/G_0$ is abelian. \\
Let $f \in L^2(G)$ and $\psi \in L^2(G)\setminus \{0\}$ be such that
\begin{flalign}
&& (m\times \mu)(\{(x,\gamma):G_\psi f(x,\gamma) \neq 0\})&<1. &\label{WQUP-step6}
\end{flalign}
For each $x \in G$, define $f_x^\psi$ as in \eqref{function}, then $f_x^\psi \in L^1(G)$. \\
Suppose that $f_x^\psi \neq 0$ for all $x\in G$. By \cite[Lemma 2.3]{Kut:03}, there exists a closed normal subgroup $H_x$ of $G$ such that $G/H_x$ is Lie and a function $g_x\in L^1(G/H)$ such that
\begin{flalign}
&& m\left(A_{f_x^\psi}\right)\ \mu\left(B_{f_x^\psi}\right)&=m_{G/H}(A_{g_x})\ \mu_{G/H}(B_{g_x}). &\label{WQUP-step7}
\end{flalign}
We note that $G/G_0H=(G/H)/(G_0H/H)$.\\
Since $G_0H/H$ is connected and open in $G/H$, we have $G_0H/H=(G/H)_0$. \\
By hypothesis, $G/G_0$ is abelian, so is $(G/H)/(G/H)_0$. \\
Thus, we can assume that $G$ is a compact Lie group. By \cite[Lemma 2.2]{Kut:03}, there exists a function $h_x\in L^1(G/G_0)$, $h_x \neq 0$ such that
\begin{flalign}
&& m\left(A_{f_x^\psi}\right)\ \mu\left(B_{f_x^\psi}\right)&\geq m_{G/G_0}(A_{h_x})\ \mu_{G/G_0}(B_{h_x}). &\label{WQUP-step8}
\end{flalign}
Since $G/G_0$ is abelian and $h_x\in L^1(G/G_0)$, $h_x \neq 0$, so by \cite{Mat:Szu:73} we have
\begin{flalign}
&&m_{G/G_0}(A_{h_x})\ \mu_{G/G_0}(B_{h_x})&\geq 1. &\label{WQUP-step9}
\end{flalign}
Combining \eqref{WQUP-step8} and \eqref{WQUP-step9}, we obtain
\begin{flalign*}
&& 1 &\leq m\left(A_{f_x^\psi}\right)\ \mu\left(B_{f_x^\psi}\right) &\\
&&&\leq m(G)\ \mu(\{\gamma: \widehat{f_x^\psi}(\gamma) \neq 0\}) &\\
&&&= \mu(\{\gamma: G_\psi f(x,\gamma) \neq 0\}), &
\end{flalign*}
for all $x \in G$. On integrating both sides with respect to $x$, we get
\begin{flalign*}
&&1&\leq \int_{G}{\mu(\{\gamma: G_\psi f(x,\gamma) \neq 0\})}\ dm(x) &\\
&&&= \int_{G}\int_{\widehat{G}}{\chi_{\{\gamma: G_\psi f(x,\gamma) \neq 0\}}(\gamma)}\ dm(x)\ d\mu(\gamma) &\\
&&&= \int_{G}\int_{\widehat{G}}{\chi_{\{(x,\gamma): G_\psi f(x,\gamma) \neq 0\}}(x,\gamma)}\ dm(x)\ d\mu(\gamma) &\\
&&&=(m\times \mu)(\{(x,\gamma): G_\psi f(x,\gamma)\neq 0\}), &
\end{flalign*}
which is a contradiction to \eqref{WQUP-step6}. So there exists $x \in G$ such that $f_x^\psi (y)=0$ for almost all $y \in G$. Since $\psi \in L^2(G)\setminus \{0\}$ is arbitrary. Thus $f=0$ a.e. 
\end{proof}
\section*{Acknowledgements}
\noindent The second author was supported by R \& D grant of University of Delhi. 

\end{document}